\documentclass[11pt]{amsart}
\usepackage{verbatim, latexsym, amssymb, amsmath,color,enumitem}
\usepackage{graphicx} 
\usepackage{hyperref} 
\usepackage{epsfig}
\usepackage{bm}
\usepackage{mathrsfs}
\usepackage[margin=1.25in]{geometry}
\usepackage{subcaption}

\newtheorem{theorem}{Theorem}[section]
\newtheorem{corollary}[theorem]{Corollary}

\newtheorem{lemma}[theorem]{Lemma}

\newtheorem{proposition}[theorem]{Proposition}

\theoremstyle{definition}
\newtheorem{definition}[theorem]{Definition}
\newtheorem{conjecture}[theorem]{Conjecture}
\newtheorem{example}[theorem]{Example}

\newtheorem{problem}[theorem]{Problem}

\newtheorem{remark}[theorem]{Remark}

\numberwithin{equation}{section}

\newcommand{\card}{\operatorname{card}}
\newcommand{\dist}{\operatorname{dist}}
\newcommand{\green}{\operatorname{Gr}}
\newcommand{\Clos}{\operatorname{Clos}}

\newcommand{\RP}{\mathbb{RP}}
\newcommand{\R}{\mathbb{R}}
\newcommand{\mB}{\mathbb{B}}

\title{Improved Hebey-Vaugon conjecture on equivariant Yamabe invariants in dimension 3}

\author{Tongrui Wang}
\address{Institute for Theoretical Sciences, Westlake Institute for Advanced Study, Westlake University, Hangzhou, Zhejiang, 310024, China}
\email{wangtongrui@westlake.edu.cn}

\author{Xuan Yao}
\address{Cornell University, Department of Mathematics, Ithaca, New York 14850}
\email{xy346@cornell.edu}

\date{September 2023}

\begin{document}

\maketitle

\begin{abstract}
    Consider a closed connected $3$-manifold $M$ acted diffeomorphically on by a compact Lie group $G$ with at least one orbit of finite cardinality. 
    We show an upper bound for the $G$-equivariant Yamabe invariant $\sigma_G(M)$ under certain topological assumptions, which improved a conjecture of Hebey-Vaugon. 
\end{abstract}

\section{Introduction}


Yamabe invariant is a smooth topological invariant that is preserved under diffeomorphisms. The Yamabe invariant of an $n$-dimensional smooth manifold $M$ is defined as the supremum of Yamabe constants taken over all conformal classes on $M$. 

More precisely, we remind the reader of the definitions of Einstein energy functional, Yamabe constant, and Yamabe invariant as follows:

\begin{definition}\label{defn:Yamabede}
Suppose $(M,g)$ is an $n$-dimensional closed Riemannian manifold, we define its Einstein-Hilbert energy functional as
\[
E(g)=\frac{\int_{M}R_gdV_g}{(\int_MdV_g)^{\frac{n-2}{n}}},
\]
where $R_g$ is the scalar curvature of $(M,g)$.

The Yamabe constant of the conformal class $[g]$ is defined as
\begin{align*}
    Y(M,[g])=\inf_{u\in H^1(M,g)}\frac{\int_{M}\frac{4(n-1)}{n-2}|\nabla u|_g^2dV_g+\int_{M}R_gdV_g}{(\int_{M}|u|^{\frac{2n}{n-2}}dV_g)^{\frac{n-2}{n}}}.
\end{align*}
Taking supremum over all conformal classes on $M$, we define the Yamabe invariant of $M$ as
\begin{align*}
    \sigma(M)=\sup_{[g]\in\mathcal C(M)}Y(M,[g]),
\end{align*}
where $\mathcal C (M):=\{[g]: [g] \text{ is a comformal class on }M\}$.
\end{definition}

The motivation behind Definition \ref{defn:Yamabede} of Yamabe constants and Yamabe invariants comes from the well-known \textit{Yamabe Problem}.

\begin{theorem}[Yamabe Problem, Yamabe,  Trudinger, Aubin, Schoen]\label{prob:Yamabe}
    Given a compact Riemannian manifold $(M^n,g)$ with dimension greater than or equal to $3$, there is a metric conformal to $g$ with constant scalar curvature.
\end{theorem}
Theorem \ref{prob:Yamabe} was first proposed by Yamabe \cite{ojm/1200689814}, and the combined work by Yamabe \cite{ojm/1200689814}, Trudinger \cite{ASNSP_1968_3_22_2_265_0}, Aubin \cite{aubin1976equations} and Schoen \cite{schoen1984conformal} gave an affirmative answer to this problem.

There is a variety of interesting topics on the Yamabe invariants of closed manifolds. One of the most important conjectures on the Yamabe invariants was made by Schoen.
\begin{conjecture}[Schoen]\label{conj:schoen}
    The Yamabe invariant of the lens space is
    \begin{align*}
        \sigma(L(p,q))=\frac{\sigma(S^3)}{p^{2/3}}:=\sigma_p,
    \end{align*}
    where $p,q$ are relatively prime.
\end{conjecture}

When $p=2$, this conjecture was solved by the celebrated work of Bray-Neves \cite{bray2004classification} via inverse mean curvature flow, and they gave a classification of $3$-manifolds with $\sigma$-invariants greater than or equal to $\mathbb{RP}^3$. Akutagawa and Neves later completed this classification \cite{akutagawa20073}. Recently, Mazurowski and the second author gave a new proof of Bray-Neves' result via harmonic function. \cite{mazurowski2023yamabe}.

However, when $p>2$, this conjecture remains open. The computation of Yamabe invariant of lens space is equivalent to considering the Yamabe invariant of $S^3$ under the corresponding isometry group control. Conjecture \ref{conj:schoen} is closely related to the Equivalent Yamabe Problem which was first proposed by Hebey and Vaugon \cite{hebey1993probleme}.

\begin{problem}[Equivariant Yamabe Problem]
    Given a Riemannian manifold $(M^n,g)$ with dimension greater than or equal to $3$, and suppose $G$ is a compact subgroup of the isometry group, find a $G$-invariant metric conformal to $g$ with constant scalar curvature.
\end{problem}

One can similarly define the equivariant Yamabe constant and Yamabe invariant as follows.

\begin{definition}[Equivariant Yamabe constant and Yamabe invariant]
  Suppose $(M^n,g)$ is a Riemannian manifold, $n\geq 3$, $G$ is a compact subgroup of the isometry group, then one can define the $G$-invariant Yamabe constant as:
  \begin{align*}
      Y_G(M,[g]):=\inf_{\tilde{g}\in [g]^G}E(\tilde{g}),
  \end{align*}
  where
  \begin{align*}
      [g]^G:=\{\tilde{g}: \tilde{g} \text{ is conformal to } g \text{ and } G\text{-invariant} \}.
  \end{align*}
  The $G$-invariant Yamabe constant is defined as
  \begin{align*}
      \sigma_G(M):=\sup_{[g]^G\in\mathcal C^G(M)}Y(M,[g]^G),
  \end{align*}
  where $\mathcal C^G(M):=\{[g]^G: [g]^G \text{ is a }G\text{-invariant conformal class on } M\}$.
\end{definition}

It is well-known that the Yamabe invariant of $S^n$ provides an upper bound for all Yamabe invariants of closed manifolds. Hebey and Vaugon \cite{hebey1993probleme} also proved similar results for the equivariant Yamabe invariant, and they gave an affirmative answer to the Equivariant Yamabe Problem.
\begin{theorem}
    [Hebey and Vaugon \cite{hebey1993probleme}]
    Suppose $(M^n,g)$ is a smooth compact Riemannian $G$-manifold, then there exists $g_0\in[g]^G$ realizing $Y(M,[g]^G)$, and
    \begin{align*}
        Y(M,[g]^G)\leq \sigma(S^n)\inf_{p\in M}\left(\card G\cdot p\right)^{\frac{2}{n}},
    \end{align*}
    where $\card(G\cdot p)$ denotes the cardinality of the orbit $G\cdot p$.
\end{theorem}

Hebey and Vaugon also conjectured that
\begin{conjecture}
    If $(M^n,g)$ is not comformal to the standard metric on $S^n$ or if the action $G$ has no fixed point, then 
    \begin{equation}\label{eq: oringinal upper bound}
        Y(M,[g]^G)< \sigma(S^n)\inf_{p\in M}\left(\card G\cdot p\right)^{\frac{2}{n}}.
    \end{equation}
\end{conjecture}

This conjecture was verified for $n\in [3,37]$ by the combined work of Hebey-Vaugon \cite{hebey1993probleme} and Madani \cite{MADANI2010241}.

In dimension $3$, we improve \eqref{eq: oringinal upper bound}, and make it more sharp.

Before stating the Main Theorem, we introduce the necessary definitions and conditions we need.
\begin{definition}\label{Def: G-connected}
    A $G$-invariant subset $\Sigma\subset M$ is said to be {\em $G$-connected} if for any two connected components $\Sigma_1,\Sigma_2$ of $\Sigma$, there exists $g\in G$ so that $\Sigma_2 = g\cdot \Sigma_1$. 

    Additionally, if $\Sigma\subset M$ is $G$-connected with every connected component given by an embedded $2$-sphere, then we say $\Sigma$ is an {\em embedded $G$-connected union of $2$-spheres}. 
\end{definition}

\begin{definition}
    A connected surface $\Sigma\subset M$ is said to be {\em separating} if $M\setminus \Sigma$ has two connected regions. 
    Otherwise, we say $\Sigma$ is {\em non-separating}. 
\end{definition}
\begin{definition}
    For a closed $3$-manifold $Q^3$, the number $\alpha(Q^3)\geq 2$ is defined as the supremum of the Euler characteristic of the smooth surface $\Sigma$ (not necessarily connected) whose complements $Q\setminus\Sigma$ have exactly two connected regions. 
\end{definition}



With the above definitions, we state our main theorem as follows, where the topological assumptions are parallel and comparable to those in \cite{bray2004classification}.

\begin{theorem}\label{Thm: main theorem}
    Suppose $M^3$ is a smooth closed manifold, $G$ is a compact Lie group acting on $M$ by diffeomorphisms. Assume further that $M$ and $G$ satisfy the following conditions:
    \begin{enumerate}[label = (\roman*)]
        \item there is at least one orbit of finite cardinality, i.e. $\inf_{p\in M} \card(G\cdot p) < \infty $;
        \item $M$ can be expressed as $P\# Q$ where $P$ is prime and $\alpha(Q)=2$;
        \item if $M$ is diffeomorphic to $S^3$, then $\inf_{p\in M} \card(G\cdot p)\geq 2$, i.e. there is no fixed point under the actions of $G$; 
        \item if $M$ contains a non-separating embedded $2$-sphere, then $\inf_{p\in M} \card(G\cdot p)\geq 2$ and every embedded $G$-connected union of non-separating $2$-spheres is connected. 
    \end{enumerate}
    Then, we have
    \begin{equation}\label{eq:improved upper bound}
        \sigma_G(M)\leq \sigma(\mathbb{RP}^3)\left(\inf_{p\in M}(\card G\cdot p)\right)^{\frac{2}{3}}.
    \end{equation}
\end{theorem}

\subsection{Acknowledgements}
The authors would like to thank Prof. Xin Zhou for his helpful discussions on this topic. The second author would like to thank Professor Scheon for his helpful discussion on this topic during his visit at Cornell University. T.W. is partially supported by China Postdoctoral Science Foundation 2022M722844. X.Y. is supported by NSF grant DMS-1945178.

\section{Blowing up model and outermost horizon}

Given $M$ and $G$ as in the main theorem, let $\mu$ be the bi-invariant Haar measure on $G$ which has been normalized to $\mu(G)=1$. 

Note the main theorem is true if $Y(M,[g_{_M}]_G)\leq 0$ for every $G$-equivariant conformal class $[g_{_M}]_G$ of metrics on $M$. 
Therefore, we assume throughout the remainder of this paper that $g_{_M}$ is a $G$-invariant metric on $M$ with $Y(M,[g_{_M}]_G) > 0$. 
By the solution of the equivariant Yamabe problem \cite{hebey1993probleme}, we can further assume without loss of generality that $g_{_M}$ minimizes the Einstein–Hilbert functional in $[g_{_M}]_G$, and thus $(M, g_{_M})$ has constant positive scalar curvature $R_M\equiv R_0>0$. 
Denote by 
\[ L_0 = \triangle_{g_{M}} -\frac{1}{8}R_0. \]
In this section, we use the $G$-equivariant Green function to blowup $M$ at the orbit with minimal cardinality. 
Then, we investigate the outermost minimal surfaces, which will be taken as the initial data of inverse mean curvature flow.  

To begin with, by assumption (i) and the compactness of $M$, we can take $p\in M$ with $\card(G\cdot p) = \inf_{q\in M}\card(G\cdot q)$. 
Let $\green_p(x)$ be the Green's function of $L_0$ at $p$ so that 
\[L_0\green_p = 0 \quad \mbox{on $M\setminus p$}\qquad \mbox{and} \qquad\lim_{q\to p} \dist_{M}(p,q)\green_p(q) =1. \]
Since $R_0>0$, such a Green’s function $\green_p$ always exists and is positive. 
Then we take a smooth $G$-invariant function $\green$ on $M\setminus G\cdot p$ defined as
\[ \green(x) := \frac{\int_G \green_{p}(g\cdot x)d\mu(g)}{\mu(G_p)}, \]
where $G_p:=\{g\in G: g\cdot p=p\}$ is the isotropy group of $p$ in $G$. 
Note $G_p$ has finite index $[G:G_p]<\infty$, which implies $\mu(G_p)>0$ and $\green$ is well defined. 
Additionally, one easily verifies that $\green$ is $G$-invariant, 
\[ L_0\green = 0\quad \mbox{on $M\setminus G\cdot p$}\qquad \mbox{and} \qquad\lim_{q\to G\cdot p} \dist_{M}(G\cdot p,q)\green(q) =1. \]
We call such a function $\green$ the $G$-invariant Green's function of $L_0$ at $G\cdot p$.

Next, consider the $G$-invariant metric 
\[g_{AF} := \green^4 \cdot g_{_M}\]
on $M\setminus G\cdot p$. 
It then follows from the choice of $\green$ that $G\cdot p$ has been sent to infinity and the scalar curvature of $g_{AF}$ is given by $R_{AF}=-8\green^{-5}L_0(\green)=0$. 
Moreover, $(M\setminus G\cdot p, g_{AF})$ is an {\em asymptotically flat Riemannian $G$-manifold}. 
Namely, there is a $G$-invariant compact set $K\subset M$ so that $M\setminus K$ is a finite union of ends $M_1,\dots,M_{k_0}$ satisfying
\begin{itemize}
    \item there is a diffeomorphism $\Phi_i: M_i\to \R^3\setminus \Clos(\mB^3_1(0)$ for each $M_i$, where $\Clos(\mB^3_1(0))$ is the standard closed unit ball;
    \item by writing the metric $g_{AF}$ as $g_{ij}$ in the local coordinate chart of $\Phi_i$, we have $g_{ij}=\delta_{ij} + O(|x|^{-1})$, $g_{ij,k}=O(|x|^{-2})$, and $g_{ij,kl}=O(|x|^{-3})$.
\end{itemize}
Note $k_0=\card(G\cdot p)<\infty$ and $\cup_{i=1}^{k_0}M_i$ is $G$-connected with each $M_i$ corresponding to a point in $G\cdot p$. 

Define then 
\[C(g_{AF}) := \inf \left\{ \frac{\int_M 8|\nabla u|^2}{(\int_M u^6)^{\frac{1}{3}}} : u\in H^1_G(M\setminus G\cdot p,g_{AF}) \mbox{ with compact support}\right\} ,\]
where $H^1_G(M\setminus G\cdot p,g_{AF})$ is the space of $G$-invariant $H^1$-Sobolev functions in $(M\setminus G\cdot p,g_{AF})$. 
Combining the constructions and $R_{AF}=0$ with a cut-off trick (cf. \cite[(3)]{bray2004classification}), we conclude 
\[Y(M,[g_{_M}]_G) = C(g_{AF}), \]
and thus it is sufficient to find a nice $G$-invariant test function $u\in H^1_G(M\setminus G\cdot p,g_{AF})$ with compact support to estimate $Y(M,[g_{_M}]_G)$. 

To construct such a $G$-invariant test function, we need to introduce some useful concepts and results of {\em outermost minimal surfaces}, which are well-known to experts. 
Firstly, define 
\[K_1:=\Clos\left( \bigcup_{\Sigma\in\mathcal{M}}\Sigma \right),\]
where $\mathcal{M}$ is the set of all smooth compact immersed minimal surfaces in $(M\setminus G\cdot p,g_{AF})$. 
Note $\mathcal{M}$, as well as $K_1$, is non-empty since $(M\setminus G\cdot p,g_{AF})$ is not topologically $\R^3$ (see (iii) in the main theorem). 
Additionally, $K_1$ is compact since there exists a foliation near infinity formed by positive mean curvature spheres. 
Moreover, because $g\cdot \Sigma\in \mathcal{M}$ provided $\Sigma\in\mathcal{M}$, we have $K_1$ is a $G$-invariant subset of $M\setminus G\cdot p$.

Next, consider the connected components $\{U_i\}_{i\in I}$ of $M\setminus (G\cdot p \cup K_1)$, and define the {\em trapped region} $K$ by 
\[K:= K_1\bigcup \left( \bigcup_{ \{ i\in I: U_i\mbox{ is bounded}\}} U_i \right),\]
which is also a non-empty $G$-invariant compact subset of $M\setminus G\cdot p$. 
Combining the $G$-invariance with a result of Meeks-Simon-Yau \cite{meeks1982embedded} (see also \cite[Lemma 4.1]{huisken2001inverse}), we conclude the topology boundary $\partial K$ of $K$ is $G$-invariant and is also a union of smooth embedded minimal $2$-spheres. 
In addition, it also follows from \cite[Lemma 4.1]{huisken2001inverse} that $M\setminus (G\cdot p\cup K)$ is $G$-connected with $k_0$ connected components $N_1,\dots,N_{k_0}$ corresponding to each end of $(M\setminus G\cdot p, g_{AF})$ (as well as each point of $G\cdot p$) such that 
\begin{itemize}
    \item[(1)] $\overline{N}_i$ is connected and asymptotically flat with compact minimal boundary;
    \item[(2)] $\overline{N}_i$ is diffeomorphic to $\R^3$ minus a finite union of open $3$-balls with disjoint closures; 
    \item[(3)] there is no other compact minimal surface in $\overline{N}_i$ except $\partial \overline{N}_i$,
\end{itemize}
where $\overline{N}_i$ is the metric completion of $N_i$. 
Therefore, we can take $\{\Sigma^{(j)}=\cup_{l=1}^{L_j}\Sigma^{(j)}_l\}_{j=1}^J$ as some disjoint embedded $G$-connected unions of minimal $2$-spheres in $(M\setminus G\cdot p, g_{AF})$ so that $\cup_{i=1}^{k_0} N_i$ is the unbounded $G$-connected component of $M\setminus\Gamma$, where
\[ \Gamma:= \cup_{j=1}^J \Sigma^{(j)} = \cup_{j=1}^J \cup_{l=1}^{L_j} \Sigma^{(j)}_l. \] 
(Note $\Sigma^{(j)}_l$ may be a non-separating sphere as a component of $K$.) 
Recall $\{\overline{N}_i\}_{i=1}^{k_0}$ is known as the set of {\em exterior regions}, and $\{\Sigma^{(j)}_l\}$ is known as the set of {\em outermost minimal $2$-spheres}. 

At the end of this section, we show $J=1$ and $L_1=k_0$ provided the topological assumptions in the main theorem. 

\begin{lemma}\label{Lem: outermost separate}
    Let $M$ and $G$ be given as in the main theorem. 
    Then, using the above notations, each $\Sigma^{(j)}_l$ separates $M\setminus G\cdot p$ into two connected components. 
    Therefore, each $\partial \overline{N}_i=\partial N_i$ is a connected embedded minimal $2$-sphere (i.e. $L_1=k_0$) and $\cup_{i=1}^{k_0}\partial \overline{N}_i$ is an embedded $G$-connected union of minimal $2$-spheres (i.e. $J=1$). 
\end{lemma}
\begin{proof}
    Suppose $\Sigma^{(1)}_1$ is non-separating. 
    Without loss of generality, we assume $\Sigma^{(1)}_1\subset \partial N_1$. 
    Then by assumption (iv) in the main theorem, $\Sigma^{(1)}=\Sigma^{(1)}_1$ is connected and $G$-invariant, i.e. $L_1=1$. 
    Since $k_0\geq 2$, we can take any other $N_i$ with $i\in\{2,\dots,k_0\}$. 
    By the $G$-connectivity of $\cup_{i=1}^{k_0}N_i$, there exists $g_i\in G$ so that $N_i=g_i\cdot N_1$ and $g_i\cdot \Sigma^{(1)}\subset \partial N_i$. 
    We then conclude from the $G$-invariance and the connectivity of $\Sigma^{(1)}$ that $k_0=2$ and $N_1$ glued with $N_2$ along $\Sigma^{(1)}$ in $M\setminus G\cdot p$. 
    In addition, one easily checks that any other connected component $\Sigma^{(j)}_1$ of $\partial N_1$ is also non-separating. 
    Hence, $L_j=1$ for every $j=1,\dots,J$, and $M\setminus G\cdot p$ is the connect sum of $N_1,N_2$ along the disjoint $G$-invariant $2$-spheres $\{\Sigma^{(j)}\}_{j=1}^J$. 
    Since $\Sigma^{(1)}$ is $2$-sided as a boundary, we can take $r>0$ sufficiently small so that the boundary of a small $r$-neighborhood $B_{r}(\Sigma^{(1)})$ of $\Sigma^{(1)}$ in $M\setminus G\cdot p$ is an embedded $G$-connected union of $2$-spheres $S_1\cup S_2$ with $S_i\subset N_i$. 
    This contradicts assumption (iv) in the main theorem as $S_1\cup S_2$ is not connected. 

    The last statement follows easily from the separating property of each $\Sigma^{(j)}_l$ and (2). 
\end{proof}


\section{Some Intuition}
In this section, we explain the intuition behind our assumptions in the Main Theorem \ref{Thm: main theorem}.

One of the most celebrated results in Yamabe invariants is the $\sigma$-classification of $3$-manifolds, which was first proven by Bray-Neves \cite{bray2004classification}, and the classification was later fully completed by Akutagawa-Neves \cite{akutagawa20073}. It is natural to conjecture that there is a corresponding classification in dimension $3$ for the equivariant Yamabe invariant.

To explain the intuition behind the assumptions of the Main Theorem \ref{Thm: main theorem}, we look into the following examples:

\begin{example}\label{ex:sphere}
    $M=S^3$, $G=\mathbb Z_2$, the action is the antipodal identification. In this case, it is obvious that $\sigma_G(M)=2^{2/3}\sigma_2$.
\end{example}
Example \ref{ex:sphere} explains the no fixed point assumption (Theorem \ref{Thm: main theorem}(iii)) when the manifold is $S^3$.

\begin{example}\label{ex:rp2timess1}
    $M=S^2\times S^1$, $G=\mathbb Z_2$, the action is the antipodal identification on $S^2$. Since this action is free, the $G$-invariant Yamabe invariant should be:
    \begin{align*}
         \sigma_G(M)=2^{2/3}\sigma(\mathbb{RP}^2\times S^1)=2^{2/3}\sigma_2.
    \end{align*}
   
\end{example}

\begin{example}\label{ex:s2timess1}
    $M=S^2\times S^1$, $G=\mathbb Z_2$, the action is the antipodal identification on $S^1$. Since the action is free, the $G$-invariant Yamabe invariant should be:
    \begin{align*}
    \sigma_G(M)=2^{2/3}\sigma(S^2\times S^1)=2^{2/3}\sigma_1.
    \end{align*}
\end{example}
Example \ref{ex:rp2timess1} and Example \ref{ex:s2timess1} are the most important cases we considered when we added our assumptions (Theorem \ref{Thm: main theorem}(iv)). The main distinction between them is whether $M/G$ contains a non-separating $2$-sphere. An appropriate assumption should be able to distinguish these two cases. 
To be exact, we now consider the $G$-invariant blowing up models of Example \ref{ex:rp2timess1} and Example \ref{ex:s2timess1}.

\begin{enumerate}
    \item When $\mathbb Z_2$ acts on $S^2$: the $\mathbb Z_2$-invariant blowing up model is an asymptotically flat manifold with $2$ symmetric ends and a $\mathbb Z_2$-invariant horizon. Since $M/G\cong\mathbb{RP}^2\times S^1$ contains no non-separating $2$-sphere, the horizon has two components, and each component is separating. See Figure \ref{Fig: 1}.
    \item When $\mathbb Z_2$ acts on $S^1$: the $\mathbb Z_2$-invariant blowing up model is an asymptotically flat manifold with $2$ symmetric ends and a $\mathbb Z_2$ horizon which contains two connected components, each of the components is not separating (since $M/G\cong S^2\times S^1$).  See Figure \ref{Fig: 2}.
\end{enumerate}
\begin{remark}
    If the $\mathbb Z_2$-action on $M$ is the reflection action on $S^2$, one gets $2^{\frac{2}{3}}\sigma_2$. If the $\mathbb Z_2$-action is the reflection action on $S^1$, one also gets
    $2^{\frac{2}{3}}\sigma_2$, since the quotient space is $S^2\times [0,1]$, whose Yamabe invariant is $\sigma_2$, see \cite{schwartz2009monotonicity}.
\end{remark}

For finite group $G$ and general $G$-manifolds, a non-separating $2$-sphere in (the smooth part of) $M/G$ shall be lifted to a $G$-connected union of non-separating $2$-spheres in $M$ with more than one connected components. 
Therefore, we proposed condition (iv) in Theorem \ref{Thm: main theorem} to eliminate the existence of non-separating $2$-spheres in (the smooth part of) $M/G$, which ensures the separating property of the $G$-invariant horizon (Lemma \ref{Lem: outermost separate}). 
As long as the horizon in blowing up models only contains separating components (e.g. Example \ref{ex:rp2timess1}), one can use the level set techniques (either inverse mean curvature flow, see \cite{bray2004classification} or the harmonic functions, see \cite{mazurowski2023yamabe}) to construct a test function and obtain the improved upper bound \eqref{eq:improved upper bound} of equivariant Yamabe invariant. For blowing up models with horizons containing non-separating components (e.g. Example \ref{ex:s2timess1}), the level set techniques no longer work.
\begin{figure}[h]
    \centering
    \begin{subfigure}{0.4\linewidth}
    \centering
        \includegraphics[width=2.5in]{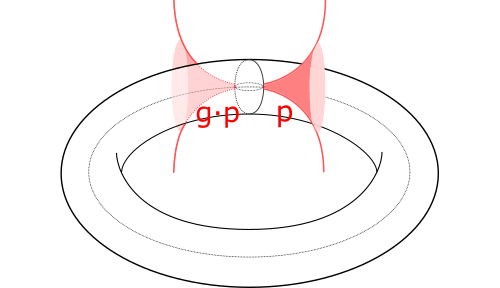} 
        \caption{$\mathbb Z_2$ acts on $S^2$}\label{Fig: 1}
    \end{subfigure}
    \hspace{1cm}
    \begin{subfigure}{0.4\linewidth}
        \centering
        \includegraphics[width=2.5in]{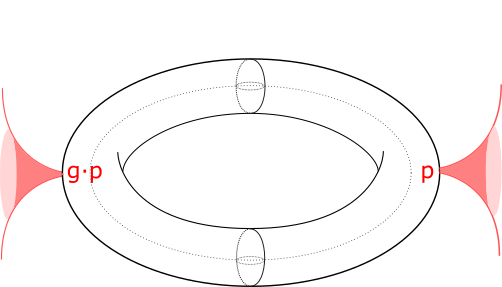}
        \caption{A non-example}\label{Fig: 2}
    \end{subfigure}
    \caption{}
\end{figure}
\section{Test function calculation}
In this section, we use harmonic functions to construct a test function as in \cite{mazurowski2023yamabe}, and complete the proof of the main theorem. 

\begin{remark}
    One can also use inverse mean curvature flow to construct the test function as in Bray-Neves' work \cite{bray2004classification} to complete this proof.
\end{remark}

    When $\inf_{p\in M}\card (G\cdot p)=\infty$, the theorem obviously holds. We only need to consider the case when $\inf_{p\in M}\card (G\cdot p):=k_0<\infty$. 

    As in the previous section, we use the $G$-equivariant Green's function to blow up $M$ at the orbit with minimal cardinality. Suppose $p=\infty_{N_1}$, consider the harmonic function defined on $N_1$:
    \begin{align}\label{harmonicfunction}
     \left\{        
     \begin{array}{ccc}
            \Delta_{g_{AF}}u=0 & x\in N_1, \\
             u=1 & x\in \partial N_1, \\
             u\to 0 & \text{ at infinity}
        \end{array}
     \right.   
    \end{align}
     As in \cite{mazurowski2023yamabe}, let $w=-\log u$ and $W(t)=\int_{\{w=t\}}|\nabla w|^2da$. By the Maximal Principle, $w$ is well-defined. 
     Note $u$ and $w$ can also be similarly defined on every other end $N_2,\dots,N_{k_0}$.

   In \cite{mazurowski2023yamabe}, Mazurowski and the second author proved the following monotonicity proposition which was essentially Theorem 7.3 in Miao's work \cite{Miao_2023}.

    \begin{proposition} Assume that $(M^3,g_{_M})$ is a complete asymptotically flat manifold with non-empty, connected boundary $\Sigma$. Assume that $M$ has non-negative scalar curvature and that $H_2(M,\Sigma)=0$. Let $u$ be the harmonic function on $M$ with $u=1$ on $\Sigma$ and $u\to 0$ at infinity. Define $w = -\log(u)$ and let $\Sigma_t = \{w=t\}$. 
Then for any $t\ge 0$ one has 
\begin{equation}
\label{monotonicity-equation}
 W(t) \le   \left[e^{-t}\sqrt{W(0)} + (1-e^{-t} )\sqrt{4\pi}\right]^2.
\end{equation}
Moreover, if equality holds for some $t > 0$, then $M$ is isometric to a spatial Schwarzschild manifold (possibly with negative mass) outside some rotationally symmetric sphere. 
\end{proposition} 

Using Corollary 7.1 in Miao's work \cite{Miao_2023}, Mazurowski and the second author obtained the following estimate, which is Corollary 3.3 in \cite{mazurowski2023yamabe}.

\begin{corollary}\label{monotonicity1}
    With the same notation and assumptions as above, assume in addition that $\Sigma$ is minimal. Then 
    \[
        W(t) \le \pi\left(2-e^{-t}\right)^2
    \]
    for all $t\ge 0$. Equality holds for some $t \ge 0$ if and only if $(M,g)$ is isometric to a spatial Schwarzschild manifold with positive mass and $\Sigma$ is the horizon. 
\end{corollary}

    Using \eqref{harmonicfunction}, we define the $G$-invariant harmonic function as $u_G=\int_G u(g\cdot x)d\mu(g)$
    then $u_G$ is a harmonic function satisfying:
    \begin{align}
       \left\{        
    \begin{array}{ccc}
            \Delta_{g_{AF}}u_G=0 & x\in N_i, \\
             u_G=1 & x\in \partial N_i,\\
             u_G\to 0 & x\to\infty_i,
    \end{array}
     \right.   
     \end{align}
    for all $i=1,2,\dots, \card (G\cdot p)$. 

Let $w_G(x):=-\log u_G(x)$, and $W_G(t):=\int_{\Sigma_t^G}|\nabla w_G|^2da$, where $\Sigma_t^G:=\{x\in M: w_G(x)=t\}$.

A direct corollary is as follows:
\begin{corollary}\label{Gmonotonicity}
    \begin{equation}
        W_G(t)\leq \card (G\cdot p)\pi(2-e^{-t})^2,
    \end{equation}
    equality holds if and only if each end is isometric to a half spatial Schwarzschild manifold with positive mass.
\end{corollary}

Now, we have all the ingredients to complete the proof of the main theorem, the following computation mostly follows from the computation in \cite{mazurowski2023yamabe}.

\begin{proof}[Proof of the main theorem]
As in \cite{mazurowski2023yamabe}, We first consider the model case, when $M=\mathbb{RP}^3$, and $g_0$ is the standard round metric. As in \cite{bray2004classification}, $(\RP^3_p,g_{AF})$ can be identified with a spatial Schwarzschild manifold with antipodal identification on the horizon. Let $w_s=-\log(u_s)$, where we use $u_s$ to 
    denote the solution to \eqref{harmonicfunction} on $(\mathbb{RP}^3_{p},g_{AF})$. Then define
\begin{align}
    f(t)=u_0(\Sigma_t^s),
\end{align}
where $\Sigma_t^s=\{x\in (\mathbb{RP}^3_p,g_{AF}): w_s(x)=t\}$, and $u_0(x)=\green(x)^{-1}$. Let $\phi_s=f\circ w_s$. As in \cite{bray2004classification},  Obata's theorem implies that $\phi_s$ is the optimal test function in the model and

\begin{align*}
\frac{\int_{\RP^3\setminus\{p\}} 8\vert \nabla \phi_s\vert^2 \, dV}{\left(\int_{\RP^3\setminus\{p\}} \phi_s^6\, dV\right)^{1/3}} = \sigma_2. 
\end{align*}

Now, consider the blow up model $(M\setminus G\cdot p,g_{AF})$, and we construct the test function
\begin{equation}
    \phi(x)=f\circ w_G(x),
\end{equation}
where $w_G(x):=-\log u_G(x)$.

By definition, $\phi(x)$ is a $G$-invariant test function, and we can compute that
\begin{equation}
  \begin{split}
    \int_{M\setminus (G\cdot p)}|\nabla \phi_G|^2dV&=\int_{M\setminus (G\cdot p)}f'(w_G(x))^2|\nabla w_G|^2dV\\
    &=\int_0^{\infty}f'(t)^2\left(\int_{\Sigma^G_t}|\nabla w_G|da\right)dt\\
    &=C_0\int_0^{\infty}f'(t)^2e^tdt.
  \end{split}
\end{equation}
And we estimate
\begin{equation}\label{dominateresitimate}
\begin{split}
    \int_{M\setminus (G\cdot p)}\phi_G^6 dV&=\int_0^{\infty}f(t)^6\left(\int_{\Sigma_t^G}|\nabla w_G|^{-1}da\right)dt\\
      &\geq \int_0^{\infty}f(t)^6(\int_{\Sigma^G_t}|\nabla w_G|^2da)^{-2}(\int_{\Sigma^G_t}|\nabla w_G|da)^3\, dt\\
&\geq C_0^3\int_0^{\infty}f(t)^6e^{3t}W_G(t)^{-2}\, dt\\
&\geq (\card (G\cdot p))^{-2}\pi^{{-2}}C_0^3\int_0^{\infty}f(t)^6e^{3t}(2-e^{-t})^{-4}\, dt.
\end{split}
\end{equation}

Therefore, we obtain
\begin{equation}
    \frac{\int_{M\setminus (G\cdot p)}8|\nabla\phi_G|^2dV}{\left(\int_{M\setminus (G\cdot p)}\phi_G^6dV\right)^{1/3}}\leq (\card (G\cdot p))^{\frac{2}{3}}\sigma_2,
\end{equation}
which implies 
\begin{equation}
    \sigma_G(M)\leq \sigma_2\left(\inf_{p\in M}\card (G\cdot p)\right)^{\frac{2}{3}},
\end{equation}
the proof is completed.
\end{proof}
\begin{remark}
    As in \cite{mazurowski2023yamabe}, one needs to use the Lebesgue Convergence Theorem for \eqref{dominateresitimate}.
\end{remark}


\bibliographystyle{abbrv}


\bibliography{reference.bib}

\end{document}